\documentclass{amsart}

\usepackage{mathtools}

\newtheorem{theorem}{Theorem}[section]
\newtheorem{lemma}[theorem]{Lemma}

\theoremstyle{definition}
\newtheorem{definition}[theorem]{Definition}

\theoremstyle{remark}

\numberwithin{equation}{section}

\begin{document}

\title{A modification of the Chang-Wilson-Wolff Inequality via the Bellman Function}

\author{Henry Riely}

\address{Department of Mathematics and Statistics, Washington State University, Pullman, Washington 99163}

\email{hriely@math.wsu.edu}

%    General info

\date{March 16, 2018.}


\begin{abstract}
We describe the Bellman function technique for proving sharp inequalities in harmonic analysis. To provide an example along with historical context, we present how it was originally used by Donald Burkholder to prove $L^p$ boundedness of the $\pm 1$ martingale transform. Finally, with Burkholder's result as a blueprint, we use the Bellman function to prove a new result related to the Chang-Wilson-Wolff Inequality.
\end{abstract}

\maketitle

\section{Introduction}

The Bellman function technique, named for applied mathematician Richard Bellman, is a tool that has been imported from the applied field of stochastic optimal control, and is now being used to tackle problems in probability and harmonic analysis. It was introduced to the world of analysis by Donald Burkholder, who in \cite{Burkholder} used it to prove that the $\pm$1 transform of a martingale is a bounded operator on $L^p$. We will borrow Burkholder's ideas to prove a new result concerning the exponential integrability of dyadic martingales.

Section 2 is expository. We briefly summarize Burkholder's use of the Bellman function to prove a sharp martingale inequality. This will serve as homage to Burkholder, the pioneer, and provide some historical context. As importantly, it will give a template upon which we will build the proof of our main result. The Bellman function technique is much easier to demonstrate than it is to describe abstractly.

In section 3, we introduce a well known inequality from harmonic analysis due to Chang, Wilson, and Wolff which classifies the order of local integrability of a function whose dyadic square function is bounded. Their result says that given a function $f:[0,1) \to \mathbb{R}$ and its dyadic square function $Sf$,
\begin{equation*}
\int_0^1 e^{f(x)-\langle f \rangle_{[0,1)}}dx \leq e^{\frac{1}{2}\|(Sf)^2(x)\|_{L^\infty}}
\end{equation*}
Next, we address a related question from \cite{SV}. Namely, we explore whether there exists a constant $\alpha$ such that
\begin{equation*}
\int_0^1 e^{f(x)-\langle f \rangle_{[0,1)}}dx \leq \int_0^1 e^{\alpha(Sf)^2(x)}dx
\end{equation*}
and if so, what the smallest valid choice of $\alpha$ is. Stated probabilistically, given a dyadic martingale with $f_0 = 0$, what is the smallest $\alpha$ such that
\begin{equation*}
\mathbb{E}e^{f_n} \leq \mathbb{E}e^{\alpha(Sf_n)^2}
\end{equation*}
We use the Bellman function to prove, without constructing an explicit example, that $\alpha = 2$ makes this inequality sharp.

Lastly, in section 4, we provide an alternate proof of our inequality by applying Cauchy-Schwarz to a result known as Rubin's lemma. We attempt to construct an example of a martingale which shows $\alpha = 2$ is sharp, but come up short. Our example pushes $\alpha$ up to $\log_2(e) \approx 1.44$, but finding the extremal martingale is left for future work.

\section{Bellman Function Technique}

We begin by illustrating the utility of the Bellman function. We shall summarize Burkholders arguments, which we will repurpose for our own problem later on. The exposition roughly follows that of \cite{Osekowski}.

\begin{definition}
Let $(\Omega, \mathcal{F}, \mathbb{P})$ be a probability space filtered by $\{\mathcal{F}_n\}$. Then the sequence of random variables $\{f_n\}$ is called a martingale if for each $n \in \mathbb{N}$
\begin{enumerate}
\item $f_n$ is $\mathcal{F}_n$ measurable\\
(i.e., $\{f_n\}$ is adpapted to the filtration $\{\mathcal{F}_n\}$)
\item $\|f_n\|_{L^1} < \infty$
\item $E[f_{n+1}|\mathcal{F}_n] = f_n$
\end{enumerate}
\end{definition}

If we replace the equality in condition 3 with $\leq$ or $\geq$, $\{f_n\}$ is called a supermartingale or submartingale respectively.

\begin{definition}
Let $\{f_n\}$ be a martingale. Define $df_0 = f_0$ and $df_n = f_n - f_{n-1}$ for $n>0$. $\{df_n\}$ is called the difference sequence of ${f_n}$.
\end{definition}

Note that $\{df_n\}$ is adapted to $\mathcal{F}_n$ and $E[df_{n+1}|\mathcal{F}_n] = 0$. It is often useful to express a martingale as the sum of its difference sequence, $f_n = \sum_{k=0}^n df_k$.

\begin{definition}
Let $\{f_n\}$ be a martingale, and define
$$g_n = \sum_{k=0}^n \epsilon_k df_k$$
where $\{\epsilon_k\}$ is a deterministic sequence, all of whose terms are $\pm1$. $\{g_n\}$ is called a $\pm 1$ transform of $\{f_n\}$.
\end{definition}

In \cite{Burkholder}, Burkholder used the Bellman function technique to prove the following result.

\begin{theorem}
There is a constant $\beta_p$ such that if $\{g_n\}$ is a $\pm 1$ transform of $\{f_n\}$, then for all $n \in \mathbb{N}$ and $0<p<\infty$
\begin{equation}
\|g_n\|_p \leq \beta_p \|f_n\|_p
\end{equation}
\end{theorem}

In other words, each $\pm 1$ transform is a bounded operator on $L^p$ for $0<p<\infty$. Theorem 1 has an interesting corollary, which highlights the connection between harmonic analysis and probability. Theorem 1 implies that the Haar basis for $L^p([0,1))$ is unconditional since the partial sums of a Haar series are a martingale on $([0,1),\mathcal{B},|\cdot|)$ (see \cite{Burkholder} or \cite{Osekowski} for details). Thinking about functions as random variables is sometimes a useful perspective in analysis as it allows us to import tools from probability where it's not obvious that they belong.

We now outline Burkholder's proof of Theorem 1 in order to demonstrate the Bellman function technique and provide a blueprint. Its essence is to relate the validity of an inequality to the existence of a special function. Often the existence of the function is easier to prove or disprove than the given inequality. We will henceforth assume that $\{f_n\}$ is a simple martingale (each $f_n$ takes on finitely many values and eventually $f_n = f_{n+1} = \dots$). Passage to the general case follows from an approximation argument.

\begin{theorem}
Suppose there exists a function $B:\mathbb{R}^2 \to \mathbb{R}$ with the following three properties.
\begin{enumerate}
\item (Majorization) $B(x,y) \geq x^p - \beta_p^p y^p \coloneqq V(x,y)$
\item (Concavity) For all $x,y,t_1,t_2 \in \mathbb{R}, \epsilon = \pm 1$, and $\alpha \in (0,1)$ such that $\alpha t_1 + (1 - \alpha) t_2 = 0$, we have
$$\alpha B(x + t_1,y + \epsilon t_1) + (1-\alpha)B(x + t_2, y + \epsilon t_2) \leq B(x,y)$$
\item (Initial condition) $B(x,\pm x) \leq 0$
\end{enumerate}
Then (2.1) holds.
\end{theorem}

As promised, we've reduced the veracity of (2.1) to the existence of a certain function with special properties. Note that, due to the majorization property, the existence of $B$ depends on our choice of $\beta$ and $p$. This is to be expected because (1) may hold for some $\beta$ and $p$ but not others.

$B$ is actually not the Bellman function, but a pointwise majorant of it. The definition of Bellman function and its relationship to our $B$ will be discussed momentarily.

Property 1 states that $B$ dominates a function $V$ whose definition is suggested by the inequality we're after. Note that (2.1) is equivalent to $\mathbb{E}V(f_n,g_n) \leq 0$. Property 2 says that $B$ is "diagonally concave," i.e., it is concave along the lines of slope $\pm 1$. This implies $\mathbb{E}B(x + \xi, y \pm \xi) \leq B(x,y)$ for all mean zero random variables $\xi$ by Jensen's inequality. The form of this concavity condition also varies with the inequality to be proven. It is chosen so that $\{B(f_n,g_n)\}$ is a supermartingale.

\begin{lemma}
If $B$ satisfies condition 2 in Theorem 2.5, then $\{B(f_n,g_n)\}$ is a supermartingale.
\end{lemma}

\begin{proof}
\begin{align}
\mathbb{E}[B(f_{n+1},g_{n+1})|\mathcal{F}_n] &= \mathbb{E}[B(f_n + df_{n+1},g_n \pm df_{n+1})|\mathcal{F}_n]\\
&\leq \mathbb{E}[B(f_{n},g_{n})|\mathcal{F}_n]\\
&= B(f_n,g_n)
\end{align}
where (2.2) uses that $g_n$ is a $\pm 1$ transform of $f_n$ and (2.3) is from applying condition 2 in Theorem 2.5 conditionally. (2.4) is because $(f_n,g_n)$ is $\mathcal{F}_n$ measurable.
\end{proof}

We now prove Theorem 2.5.

\begin{proof}
Recall that it suffices to show $\mathbb{E}V(f_n,g_n) \leq 0$.
\begin{align}
\mathbb{E}V(f_n,g_n) &\leq \mathbb{E}B(f_n,g_n)\\
&=\mathbb{E}\big[\mathbb{E}[B(f_n,g_n)|\mathcal{F}_{n-1}]\big]\\
&\leq \mathbb{E}B(f_{n-1},g_{n-1})
\end{align}
where (2.5) follows from majorization (2.6) is because conditional expectation preserves expectation. (2.7) is from the fact that $\{B(f_n,g_n)\}$ is a supermartingale.
Repeating the argument n times, we get
$$\mathbb{E}V(f_n,g_n) \leq \mathbb{E}B(f_0,g_0) = \mathbb{E}B(df_0,\pm df_0) \leq 0$$
The final inequality is from the initial condition.
\end{proof}

With Theorem 2.5 proved, we can now prove Theorem 2.4 by producing an appropriate $B$. However, we can actually do more. In fact, given a $0<p<\infty$ and $\beta_p$, if no such $B$ exists, then (2.1) is false.

\begin{theorem}
Given $0<p<\infty$ and $\beta_p$, (2.1) holds if and only if there exists a function $B$ with the three properties from Theorem 2.
\end{theorem}

\begin{proof}
We know that if $B$ exists, then (2.1) holds by Theorem 2.5. Now we must show that if (2.1) holds, then $B$ exists.

Let $\mathcal{M}(x,y)$ be the set of all $\mathbb{R}^2$ valued martingales $(f_n,g_n)$ such that $(f_0,g_0) \equiv (x,y)$ and $dg_n = \pm df_n$ for $n \geq 1$. We define the Bellman function

$$\mathcal{B}(x,y) \coloneqq \sup\{\mathbb{E}V(f_n,g_n) : (f_n,g_n)\in M(x,y)\}$$

We will show that $\mathcal{B}$ is the desired function $B$ possessing the three properties. Majorization is straightforward. Observe that the deterministic pair $(x,y) \in \mathcal{M}(x,y)$. The initial condition $\mathcal{B}(x,\pm x) \leq 0$ follows from (2.1). To show concavity, we use a "splicing argument." Take $x,y,t_1,t_2,\alpha,\epsilon$ as in the statement of the concavity condition. Choose any $(f_n^a,g_n^a) \in \mathcal{M}(x + t_1, y + \epsilon t_1)$ and $(f_n^b,g_n^b) \in \mathcal{M}(x + t_2, y + \epsilon t_2)$. We may assume the pairs are given on $([0,1),\mathcal{B},|\cdot|)$. We will define another martingale by "splicing" these two. Let $(f_0,g_0) \equiv (x,y)$ and for $n \geq 0$
\[ (f_{n+1},g_{n+1})(\omega) = 
	\begin{cases} 
		(f_n^a,g_n^a)(\frac{\omega}{\alpha}) & \omega \in [0,\alpha) \\
		(f_n^b,g_n^b)(\frac{\omega - \alpha}{1 - \alpha}) & \omega \in [\alpha,1) \\
	\end{cases}
\]
One can check that $\{(f_n,g_n)\} \in \mathcal{M}(x,y)$, and
\begin{align*}
\mathcal{B}(x,y) \geq \mathbb{E}V(f_n,g_n) = \alpha\mathbb{E}V(f_n^a,g_n^a) + (1 - \alpha)\mathbb{E}V(f_n^b,g_n^b)
\end{align*}
Taking the supremum over all such $(f_n^a,g_n^a)$ and $(f_n^b,g_n^b)$ yields
\begin{align*}
\mathcal{B}(x,y) \geq \alpha \mathcal{B}(x + t_1,y + \epsilon t_1) + (1-\alpha)\mathcal{B}(x + t_2, y + \epsilon t_2)
\end{align*}

The last thing that must be checked is that $\mathcal{B}$ is finite on all of $\mathbb{R}^2$. We know $\mathcal{B} \geq V > -\infty$ so we only need to check $\mathcal{B}(x,y) < \infty$. The initial condition says that $\mathcal{B}(x,\pm x) \leq 0$. Now suppose $|x|\neq |y|$ and let $(f_n,g_n)\in \mathcal{M}(x,y)$. We construct another martingale $(f_n',g_n')$ as follows (again we may assume our martingales are defined on $([0,1),\mathcal{B},|\cdot|)$.

\begin{align*}
(f'_0,g'_0) &\equiv (\frac{x+y}{2},\frac{x+y}{2}) \\
(f'_1,g'_1)(\omega) &=
	\begin{cases}
		(x,y) & \omega\in [0,\frac{1}{2}) \\
		(y,x) & \omega\in [\frac{1}{2},1) \\
	\end{cases} \\
(f'_n,g'_n)(\omega) &=
	\begin{cases}
		(f_{n-1},g_{n-1})(2\omega) & \omega\in [0,\frac{1}{2}) \\
		(y,x) & \omega\in [\frac{1}{2},1)\\
	\end{cases}
\end{align*}
\noindent
where the last equality holds for $n \geq 2$. Note that $f'$ is a $\pm 1$ transform of $g'$ and 
$$0 \geq \mathbb{E}V(f_n',g_n') = \frac{1}{2}V(y,x) + \frac{1}{2}\mathbb{E}V(f_{n-1},g_{n-1})$$
\noindent
Taking the supremum over all $(f_n,g_n)\in \mathcal{M}(x,y)$ gives $\mathcal{B}(x,y) \leq -V(y,x) < \infty$ and we're done.

\end{proof}

Equipped with Theorem 2.7, we can do more than just prove our inequality. We can find the optimal constant $\beta_p$. If $B$ does not exist when $\beta_p < C_p$ but $B$ does exist when $\beta_p \geq C_p$, then $C_p$ is optimal. We will not construct $B$ here because our intention is not to reproduce Burkholder's result, but rather demonstrate the utility of his approach and provide a blueprint for proving our new inequality.

\section{A modification of the Chang-Wilson-Wolff Inequality}

Let $f \in L^1(I)$ for some real interval $I$. Let $\mathcal{F}_n$ be the $\sigma$ algebra generated by the dyadic subintervals of $I$ of length $|I|2^{-n}$ and let $\mathcal{F} = \bigcup \mathcal{F}_n$. Then $f_n = \mathbb{E}[f|\mathcal{F}_n]$ is a martingale on the probability space $(I,\mathcal{F},\frac{|\cdot|}{|I|})$. Due to the dyadic filtration, such a martingale is called a dyadic martingale.

Now define $Sf_n = \|\{df_n\}\|_{\ell^2} = \sqrt{\sum_{k=1}^n (df_k)^2}$ and let $Sf = \lim\limits_{n \to \infty} Sf_n$. In the context of probability, the sequence $\{Sf_n\}$ is called the quadratic variation of the martingale $\{f_n\}$. In the Littlewood-Paley theory of harmonic analysis, $Sf$ is called the dyadic square function of $f$ and is the discrete counterpart of the Lusin area function, whose definition would take us too far from our present goal. For a description of the connection between the square function and the Luzin area function, refer to \cite{Llorente}. For a survey of the role of square functions in harmonic analysis, consult \cite{Stein}.

In \cite{CWW}, the authors answer a question posed by the illustrious harmonic analyst Elias Stein. The question concerns the size of the functions in a certain subspace of BMO: What is the sharp order of local integrability of a function $f$ with a pointwise bounded square function $Sf$? Suppose for simplicity and concreteness that $f \in L^1([0,1))$. The authors show that
\begin{equation}
\int_0^1 e^{f(x)-\langle f \rangle_{[0,1)}}dx \leq e^{\frac{1}{2}\|Sf\|^2_{L_{[0,1)}^\infty}}
\end{equation}
This result, called the Chang-Wilson-Wolff inequality, can be used to show that if $Sf$ is bounded, then $f$ is exponentially square integrable \cite{Pipher}.

In \cite{SV}, the authors explore a related question: Is there a constant $\alpha$ such that 

\begin{equation}
\int_0^1 e^{f(x)-\langle f \rangle_{[0,1)}}dx \leq \int_0^1 e^{\alpha(Sf)^2(x)}dx
\end{equation}
Inspired by their work, our goal in this section is to use the Bellman function technique to find the smallest such $\alpha$. In the language of probability, given a dyadic martingale with $f_0 = 0$, is there a constant $\alpha$ such that for each $n \in \mathbb{N}$
\begin{equation}
\mathbb{E}e^{f_n} \leq \mathbb{E}e^{\alpha(Sf_n)^2}
\end{equation}
If so, what is the smallest such $\alpha$? We shall use Burkholder's approach as a blueprint. As before, the veracity of (3.3) can be recast as a question of the existence of a special function satisfying three properies. The nature of these properties are gleaned from the details of the inequality we are after.

\begin{theorem}
(3.3) holds if and only if there exists a function $B:\mathbb{R} \times [0,\infty) \to \mathbb{R}$ with the following three properties.
\begin{enumerate}
\item (Majorization) $B(x,y) \geq e^x-e^{\alpha y} \coloneqq V(x,y)$
\item (Concavity) $\frac{B(x + \delta, y + \delta^2) + B(x - \delta, y + \delta^2)}{2} \leq B(x,y)$ for any $\delta \in \mathbb{R}$
\item (Initial condition) $B(0,0) \leq 0$
\end{enumerate}
\end{theorem}

\noindent
\textbf{Remark:} As before, property 2 has a probabilistic interpretation. Given any discrete random variable $\xi$ with $\mathbb{P}(\xi = \delta) = \mathbb{P}(\xi = -\delta) = \frac{1}{2}$, we have $\mathbb{E}B(x + \xi, y + \xi^2) \leq B(x,y)$.

\begin{proof}
The proof is the same as those of Theorems 2.5 and 2.7 in the previous section, only adapted to the present inequality. Keep in mind that we are assuming throughout $\{f_n\}$ is a simple dyadic martingale with $f_0=0$. First, assume $B$ exists. Then $B(f_n,(Sf)^2_n)$ is a supermartingale. Indeed,
\begin{align}
\mathbb{E}[B(f_{n+1},(Sf)^2_{n+1})|\mathcal{F}_n] &= \mathbb{E}[B(f_n + df_{n+1}, (Sf)^2_n + (df)^2_{n+1})|\mathcal{F}_n]\\
&\leq \mathbb{E}[B(f_{n},g_{n})|\mathcal{F}_n]\\
&= B(f_n,g_n)
\end{align}
where (3.5) is from applying the concavity property. Therefore, we have
\begin{align}
\mathbb{E}V(f_n,(Sf)^2_n) &\leq \mathbb{E}B(f_{n},(Sf)^2_{n})\\
&= \mathbb{E}\big[\mathbb{E}[B(f_{n},(Sf)^2_{n})|\mathcal{F}_{n-1}]\big]\\
&\leq \mathbb{E}B(f_{n-1},(Sf)^2_{n-1})
\end{align}
where (3.7) follows from majorization and (3.9) from the fact that $\{B(f_n,(Sf)^2_n)\}$ is a supermartingale.
Repeating the argument n times, we get
$$\mathbb{E}V(f_n,(Sf)^2_n) \leq \mathbb{E}B(f_0,(Sf)^2_0) = \mathbb{E}B(0,0) \leq 0$$
To prove the other direction, we assume (3.3) holds, and construct $B$. As before, we define the Bellman function $\mathcal{B}(x,y) \coloneqq \sup\{\mathbb{E}V(f_n,(Sf)^2_n) : (f_n,(Sf)^2_n)\in \mathcal{M}(x,y)\}$ where $\mathcal{M}(x,y)$ is the set of all $\mathbb{R} \times [0,\infty)$ valued processes $(f_n,(Sf)^2_n)$ such that $\{f_n\}$ is a dyadic martingale with $f_0 = x$ and $(Sf)^2_n = y + \sum_{k=1}^n (df_k)^2$ where $df_0 \coloneqq 0$. Again, as before, we show that $\mathcal{B}$ satisfies the three properties.

Majorization follows from observing the constant pair $(x,y) \in \mathcal{M}(x,y)$ and the initial condition follows from (3.3). Once again, we can get concavity by a splicing argument.

Choose any $(f_n^a,(Sf_n^a)^2) \in \mathcal{M}(x + \delta, y + \delta^2)$ and $(f_n^b,(Sf_n^b)^2) \in \mathcal{M}(x - \delta, y + \delta^2)$. We may assume the pairs are given on $([0,1),\mathcal{B},|\cdot|)$. We will define another martingale by "splicing" these two. Let $(f_0,(Sf_0)^2) \equiv (x,y)$ and for $n \geq 0$
\[ (f_{n+1},(Sf_{n+1})^2)(\omega) = 
	\begin{cases} 
		(f_n^a,(Sf_n^a)^2)(2\omega) & \omega \in [0,\frac{1}{2}) \\
		(f_n^b,(Sf_n^b)^2)(2(\omega - \frac{1}{2})) & \omega \in [\frac{1}{2},1) \\
	\end{cases}
\]
One can check that $\{(f_n,(Sf_n)^2)\} \in \mathcal{M}(x,y)$, and
\begin{align*}
\mathcal{B}(x,y) \geq \mathbb{E}V(f_n,(Sf_n)^2) = \frac{\mathbb{E}V(f_n^a,(Sf_n^a)^2) + \mathbb{E}V(f_n^b,(Sf_n^b)^2)}{2}
\end{align*}
Taking the supremum over all such $(f_n^a,(Sf_n^a)^2)$ and $(f_n^b,(Sf_n^b)^2)$ yields
\begin{align*}
\frac{\mathcal{B}(x + \delta, y + \delta^2) + \mathcal{B}(x - \delta, y + \delta^2)}{2} \leq \mathcal{B}(x,y)
\end{align*}
\noindent
Finally, we must show $\mathcal{B}$ is finite on $\mathbb{R}\times [0,\infty)$. We know $-\infty < V \leq \mathcal{B}$, so we only need to show $\mathcal{B} < \infty$. It follows from the definition of $\mathcal{B}$ that $\mathcal{B}(x + \delta, y + \frac{\delta}{\alpha}) = e^\delta \mathcal{B}(x,y)$ for any $\delta \in \mathbb{R}$ such that $y + \frac{\delta}{\alpha} \geq 0$. Hence it suffices to show that $\mathcal{B}$ is finite along the x axis. Using the concavity property, we have

\begin{align*}
\mathcal{B}(x,0) &\geq \frac{\mathcal{B}(x + \delta, \delta^2) + \mathcal{B}(x - \delta, \delta^2)}{2}\\
&= \frac{e^{\alpha\delta^2}\mathcal{B}(x + \delta - \alpha\delta^2, 0) + e^{\alpha\delta^2}\mathcal{B}(x - \delta - \alpha\delta^2, 0)}{2}
\end{align*}
\noindent
Therefore,
\begin{align*}
\mathcal{B}(x+\delta-\alpha\delta^2, 0) &\leq 2e^{-\alpha\delta^2}\mathcal{B}(x,0) - \mathcal{B}(x-\delta-\alpha\delta^2, 0)\\
&\leq 2e^{-\alpha\delta^2}\mathcal{B}(x,0) - V(x-\delta-\alpha\delta^2, 0)
\end{align*}
\noindent
The initial condition says that $\mathcal{B}(0,0) \leq 0,$ and we can see from the definition that $\mathcal{B}$ is non-decreasing in the positive x direction, so $\mathcal{B}(x,0) \leq 0$ for all $x \leq 0$. Furthermore, this monotonicity means it suffices to show $\mathcal{B}(x_n,0) < \infty$ for some sequence $x_n \to \infty$. If we take $\delta$ small enough that $\delta - \alpha\delta^2 > 0,$ we have 
\begin{align*}
\mathcal{B}(\delta-\alpha\delta^2, 0) &\leq 2e^{-\alpha\delta^2}\mathcal{B}(0,0) - V(-\delta-\alpha\delta^2, 0) < \infty \\
\mathcal{B}(2(\delta-\alpha\delta^2), 0) &\leq 2e^{-\alpha\delta^2}\mathcal{B}(\delta-\alpha\delta^2,0) - V(-2\alpha\delta^2, 0) < \infty\\
\mathcal{B}(3(\delta-\alpha\delta^2), 0) &\leq 2e^{-\alpha\delta^2}\mathcal{B}(2(\delta-\alpha\delta^2),0) - V(\delta - 3\alpha\delta^2, 0) < \infty\\
\dots\\
\mathcal{B}(n(\delta-\alpha\delta^2), 0) &\leq 2e^{-\alpha\delta^2}\mathcal{B}((n-1)(\delta-\alpha\delta^2),0) - V((n-2)\delta - n\alpha\delta^2, 0) < \infty\\
\end{align*}

\end{proof}

\subsection*{Searching for a Bellman Function Candidate}

What remains is to find a function $B$ so that we may apply Theorem 3.1. Our reasoning along the way needn't be rigorous or even correct because once we arrive at a candidate $B$, we will prove that it has the three properties. The purpose of this section is to demonstrate how one might approach the task of searching for $B$. Our approach will be to build a PDE. It should be noted that, a priori, we don't even know that a differentiable $B$ exists, but the purpose of this section is not to rigorously prove anything. Rather, it's to explain how one might arrive at a candidate, which we can then test for the three required properties.

The concavity condition $\frac{B(x - \delta, y + \delta^2) + B(x + \delta, y + \delta^2)}{2} \leq B(x,y)$ suggests that $B$ is concave along parabolic paths. Thus, we expect that $\frac{d^2}{d\delta^2}B(x + \delta, y + \delta^2) \leq 0$.
\begin{align*}
\frac{d^2}{d\delta^2}B(x + \delta, y + \delta^2) &= B_{xx}(x+\delta, y + \delta^2) + B_{xy}(x+\delta, y + \delta^2)4\delta\\
&+ B_{yy}(x+\delta, y + \delta^2)2\delta^2 + 2B_y(x+\delta, y + \delta^2)\\
&\leq 0
\end{align*}
Evaluating at $\delta = 0$ yields
\begin{equation}
B_{xx}(x,y) + 2B_y(x,y) \leq 0
\end{equation}
So $B$ is a subsolution of the reverse heat equation.

A common strategy is to search for Bellman candidates which share certain homogeneity properties with $V$. In our case, we have
$$V\big(x + \delta,y + \frac{\delta}{\alpha}\big) = e^{x+\delta} - e^{\alpha(y + \frac{\delta}{\alpha})} = e^\delta V(x,y)$$
Let $\delta = -\alpha y$. Then
$$V(x - \alpha y, 0) = e^{-\alpha y}V(x,y)$$
Note that functions with this property are completely determined by their values along the x axis. We shall search for a $B$ that shares this property.\\
Define $f(x) = B(x, 0)$ and assume $B$ has the form
\begin{equation}
B(x,y) = e^{\alpha y}f(x-\alpha y)
\end{equation}
Plugging (3.11) into (3.10) and replacing the inequality with an equation gives
\begin{equation}
e^{\alpha y}f''(x-\alpha y) - 2\alpha e^{\alpha y} f'(x-\alpha y) +2\alpha e^{\alpha y} f(x-\alpha y) = 0
\end{equation}
Dividing through by $e^{\alpha y}$ reveals a second order ODE with constant coefficients. The solutions to the characteristic equation are $\alpha \pm \sqrt{\alpha^2 -2\alpha}$. If $\alpha < 2$ then these solutions are complex, meaning $f$ oscillates. Thus $B(x,0)=f(x)$ will have no hope of dominating $V(x,0) = e^x - 1$. On the other hand, if $\alpha = 2$, then $f(x) = C_1xe^{2x} + C_2e^{2x}$. We must select $C_1$  and $C_2$ such that $f(x) = B(x,0) \geq V(x,0) = e^x - 1$ (majorization) and $f(0) = B(0,0) \leq V(0,0) = 0$ (initial condition). Given these constraints, we find $f(x)=xe^{2x}$ and so our Bellman function candidate is $B(x,y) = (x - 2y)e^{2x - 2y}$. Unfortunately, this is not a valid $B$ as the concavity property is violated. For example, $\frac{B(1 - \delta, 1 + \delta^2) + B(1 + \delta, 1 + \delta^2)}{2} \to 0$ as $\delta \to \infty$ but $B(1,1) = -1$. Interestingly, although this line of thinking produced the wrong $B$, it produced the correct $\alpha$. It turns out $2$ is the sharp constant in (3.3).

One problem is that, assuming $B$ is smooth, (3.10) is a necessary but insufficient condition for concavity along parabolic paths. It only tests for concavity at the vertex of each path. Also, turning our differential inequality
\begin{equation}
f'' - 2\alpha f' +2\alpha f \leq 0
\end{equation}
into an equation was unjustified even though doing so produced the optimal $\alpha$.

At this point, there is no shame in resorting to guess and check. That is, looking for solutions to (3.13) with $f(0) = 0$ and $f'(0)=1$ (to ensure compliance with the majorization property and the initial condition) and testing the corresponding $B(x,y)$ for the concavity property. $f(x) = e^{2x} - e^x$ is such a function (for $\alpha = 2$) and the associated Bellman candidate is $B(x,y) = e^{2x - 2y} - e^x$. If we prove this $B$ has the three desired properties, we shown (3.3) holds with $\alpha = 2$.

\begin{theorem}
Suppose $\{f_n\}$ is a dyadic martingale with $f_0=0$. Then
$$\mathbb{E}e^{f_n} \leq \mathbb{E}e^{2(Sf_n)^2}$$
\end{theorem}

\begin{proof}
$B(x,y) = e^{2x - 2y} - e^x$ obeys the three properties from Theorem 3.1. We begin with majorization. Using the inequality $1 - e^t \leq e^{-t} - 1$ we have
\begin{align*}
V(x,y) = e^x - e^{2y} = e^x(1 - e^{2y - x}) \leq e^x(e^{x - 2y} - 1) = B(x,y)
\end{align*}
Next, we show that $B$ satisfies the concavity condition.
\begin{align*}
&\frac{B(x+\delta,y+\delta^2)+B(x-\delta,y+\delta^2)}{2}\\
&= \frac{e^{2(x+\delta) - 2(y+\delta^2)} - e^{x+\delta} + e^{2(x-\delta) - 2(y+\delta^2)} - e^{x-\delta}}{2}\\
&= e^{2x - 2y - 2\delta^2}\cosh2\delta - e^x\cosh\delta\\
&\leq e^{2x - 2y} - e^x = B(x,y)
\end{align*}
In the last line, we use the inequality $1 \leq \cosh t \leq e^\frac{t^2}{2}$.
Lastly, the initial condition is immediate. $B(0,0) = 0$.
\end{proof}

At this point, we've answered half of our original question: (3.3) holds with $\alpha = 2$. To show that $\alpha = 2$ is minimal, we must show that the three properties in Theorem 3.1 are mutually incompatible when $\alpha < 2$. Then Theorem 3.1 implies (3.3) is false for this range of $\alpha$.

Our strategy will be to prove that the Bellman function cannot possess the three properties assuming it is twice continuously differentiable. Then we will drop this assumption using a mollification argument. This scheme will first require a couple of lemmas about a class of differential inequalities. \footnote{The author thanks Iosif Pinelis for his assistance provided via \texttt{mathoverflow.net} in the proof of these lemmas.}

\begin{lemma}
If $g:\mathbb{R} \to \mathbb{R}$ satisfies $g'' \leq -bg$ where $b$ is some positive constant and $g(x_0) > 0$, then there is some $x > x_0$ such that $g(x) = 0$.
\end{lemma}

\begin{proof}
Of course, by the intermediate value theorem, it suffices to show that eventually $g(x) \leq 0$. Suppose for contradiction that $g(x) > 0$ for all $x>x_0$. Then on $[x_0,\infty)$, $g'$ is monotonically decreasing because $g'' \leq -bg < 0$. Thus $\lim\limits_{x \to \infty}g'(x)$ exists and is either finite or $-\infty$. Call this limit $L$.\\

If $L<0$, then we are done because $g'(x)$ is eventually bounded above by $L+\epsilon < 0$ and thus $g(x) \to -\infty$.\\

On the other hand, suppose that $L \geq 0$. Then $g'(x)\geq 0$ for all $x \geq x_0$, and thus for such $x$ we have $g(x) \geq g(x_0)$.\\

Now fix $\epsilon \in \left(0,\frac{bg(x_0)}{2}\right)$. Select $c$ such that $|g'(x) - L| < \epsilon$ when $x>c$. Lastly, choose $x_1$ and $x_2$ such that $c<x_1<x_2$ and $x_2 - x_1 > 2$. Then for some $t \in (x_1, x_2)$ we have

\begin{align*}
g''(t) = \frac{g'(x_2) - g'(x_1)}{x_2 - x_1} &\geq \frac{L-\epsilon - (L + \epsilon)}{x_2 - x_1}\\
&= \frac{-2\epsilon}{x_2 - x_1}\\
&\geq -\epsilon\\
&\geq \frac{-bg(x_0)}{2}
\end{align*}

Therefore,
$$\frac{-bg(x_0)}{2} \leq g''(t) \leq -bg(t) \leq -bg(x_0)$$
which gives our contradiction.

\end{proof}

\begin{lemma}
Suppose $g:\mathbb{R} \to \mathbb{R}$ satisfies $g'' \leq -bg$ where $b$ is some real positive constant. If $g(x_0) > 0$ then $g(x) = 0$ for some $x \in \big[x_0, x_0 + \frac{\pi}{\sqrt{b}}\big]$.
\end{lemma}

\begin{proof}
We break into two cases.\\

Case 1: $g'(x_0) \leq 0$\\
Let $x_1 = \min\{x>x_0: g(x)=0\}$. We know $x_1$ exists by the previous lemma. Note that $g''\leq -bg \leq 0$ on $[x_0,x_1]$ since $g\geq 0$ there. Hence $g'\leq 0 $ on $ [x_0,x_1]$ because $g'(x_0)\leq 0$ and $g'$ is decreasing.
Let $E(x) = g'(x)^2 + bg(x)^2$. Then for all $x \in [x_0,x_1]$
\begin{align*}
E'(x) &= 2g'(x)g''(x) + 2bg(x)g'(x)\\
&= 2g'(x)(g''(x) + bg(x))\\
&\geq 0
\end{align*}
Therefore $E(x_0)\leq E(x)$ for all $x\in[x_0,x_1]$, i.e., $g'(x_0)^2 + bg(x_0)^2 \leq g'(x)^2 + bg(x)^2$. Therefore, recalling that $g'(x)\leq$0, we have
$$1 \leq \frac{-g'(x)}{\sqrt{g'(x_0)^2 + bg(x_0)^2 - bg(x)^2}}$$
Integrating both sides over the interval $[x_0,x_1]$ yields
\begin{align*}
x_1 - x_0 &\leq \left.\frac{-1}{\sqrt{b}}\sin^{-1}\left(g(x)\sqrt{\frac{b}{g'(x_0)^2 + bg(x_0)^2}}\right)\right\vert_{x=x_0}^{x=x_1}\\
&= \frac{-1}{\sqrt{b}}\left[\sin^{-1}(0) - \sin^{-1}\left(\sqrt{\frac{bg(x_0)^2}{g'(x_0)^2 + bg(x_0)^2}}\right)\right]\\
&\leq \frac{-1}{\sqrt{b}}(0 - \sin^{-1}(1))\\
& = \frac{\pi}{2\sqrt{b}}
\end{align*}

Case 2: $g'(x_0) > 0$\\
This time we let $x_2 = \min\{x>x_0: g(x)=0\}$. Again we note that $g''\leq 0$ on $[x_0,x_2]$. Therefore, there is a point $x_1 \in [x_0,x_2]$ such that $g'(x_1) = 0$ with $g'>0$ on $[x_0,x_1]$ and $g'<0$ on $[x_1,x_2]$. From Case 1, we know that $x_2 - x_1 \leq \frac{\pi}{2\sqrt{b}}$. It suffices to show $x_1 - x_0 \leq \frac{\pi}{2\sqrt{b}}$.\\

Recall $E(x) = g'(x)^2 + bg(x)^2$ and $E'(x) = 2g'(x)(g''(x) + bg(x))$. Since $g' \geq 0$ on $[x_0,x_1]$, $E'(x) \leq 0$ there. Thus on that interval, we have $g'(x)^2 + bg(x)^2 \geq g'(x_1)^2 + bg(x_1)^2 = bg(x_1)^2$ which implies
$$1 \leq \frac{g'(x)}{\sqrt{bg(x_1)^2 - bg(x)^2}}$$
As before, we integrate both sides over $[x_0,x_1]$ to get
\begin{align*}
x_1 - x_0 &\leq \left.\frac{1}{\sqrt{b}}\sin^{-1}\left(\frac{g(x)}{g(x_1)}\right)\right\vert_{x=x_0}^{x=x_1}\\
&= \frac{1}{\sqrt{b}}\left[\sin^{-1}(1) - \sin^{-1}\left(\frac{g(x_0)}{g(x_1)}\right)\right]\\
&\leq \frac{\pi}{2\sqrt{b}}
\end{align*}

\end{proof}

\begin{theorem}
If $\mathcal{B}(x,0) \in \mathcal{C}^2(\mathbb{R})$, then $\alpha = 2$ is the minimal constant such that (3.3) holds. Here $\mathcal{B}$ is the theoretical Bellman function as in the the proof of Theorem 3.1.
\end{theorem}

\begin{proof}
Recall that $\mathcal{B}(x,y) = \sup\{\mathbb{E}V(f_n,(Sf)^2_n) : (f_n,(Sf)^2_n)\in \mathcal{M}(x,y)\}$ where $\mathcal{M}(x,y)$ is the set of all $\mathbb{R} \times [0,\infty)$ valued processes $(f_n,(Sf)^2_n)$ such that $\{f_n\}$ is a dyadic martingale with $f_0 = x$ and $(Sf)^2_n = y + \sum_{k=1}^n (df_k)^2$. By the proof of Theorem 3.1, $\mathcal{B}$ satisfies the majorization and concavity properies as well as the initial condition as long as (3.3) holds. We will show that if $\alpha < 2$, $\mathcal{B}$ does not obey these properties, and thus (3.3) is false.

From the definition of $\mathcal{B}$, we have $$e^{\alpha y}\mathcal{B}(x - \alpha y, 0) = \mathcal{B}(x,y)$$
This is essentially a consequence of the fact that $V$ has this property. Thus, as before, if we define $f(x) = \mathcal{B}(x,0)$ then $\mathcal{B}$ has the form
\begin{equation}
\mathcal{B}(x,y) = e^{\alpha y}f(x - \alpha y)
\end{equation}
Plugging (3.13) into the concavity property and setting $y=0$ gives
\begin{equation}
\frac{f(x - \delta -\alpha \delta^2) + f(x + \delta -\alpha \delta^2)}{2} \leq e^{-\alpha\delta^2}f(x)
\end{equation}
If we express $f$ as a second order Taylor polynomial in $\delta$ centered at $x$ (valid because we are assuming $f\in\mathcal{C}^2$), then (3.14) becomes
\begin{align*}
f(x) - \alpha\delta^2 f'(x) + \frac{\delta^2}{2} f''(x) + O(\delta^3) &\leq e^{-\alpha\delta^2}f(x)\\
&= f(x) - \alpha\delta^2 f(x) + O(\delta^4)
\end{align*}
Dividing through by $\frac{\delta^2}{2}$ and letting $\delta \to 0$, we get
\begin{equation}
f''(x) - 2\alpha f'(x) +2\alpha f(x) \leq 0
\end{equation}
Multiply both sides by $e^{-\alpha x}$ to get
\begin{equation*}
e^{-\alpha x}f''(x) - 2\alpha e^{-\alpha x}f'(x) +2\alpha e^{-\alpha x}f(x) = \left(e^{-\alpha x}f(x) \right)'' + (2\alpha - \alpha^2)e^{-\alpha x}f(x) \leq 0
\end{equation*}
Letting $g(x) \coloneqq e^{-\alpha x}f(x)$ and $b \coloneqq 2\alpha - \alpha^2$ we have $g'' + bg \leq 0$.

Suppose\footnote{Assuming $\alpha > 0$ is valid because if (3.3) fails for some $\alpha$, it clearly fails for all smaller $\alpha$.} $0 < \alpha < 2$ so that $b > 0$. Then by Lemma 3.3, $g(x) \leq 0$  and thus $f(x) \leq 0$ for some $x>0$. Hence, for this $x$, $\mathcal{B}(x,0) = f(x) \leq 0 < V(x,0) = e^x - 1$. Therefore, for this range of $\alpha$, the Bellman function $\mathcal{B}$ cannot posses both the majorization and concavity properties of Theorem 3.1, and (3.3) is false.

\end{proof}

The next theorem allows us to drop the smoothness assumption on $\mathcal{B}$. However, we will use a mollification argument which requires $\mathcal{B}$ to be continuous. We will subsequently show that the continuity of $\mathcal{B}$ is a consequence of the concavity property.

\begin{theorem}
If $\mathcal{B}(x,0)$ is continuous, then $\alpha = 2$ is the minimal constant such that (3.3) holds.
\end{theorem}

\begin{proof}
Let $\eta(x)$ be the standard mollifier.
$$
\eta(x) \coloneqq \begin{cases}
Ce^{\frac{1}{x^2-1}} & \text{if } |x| < 1\\
0 & \text{if } |x| \geq 1
\end{cases}
$$
with $C>0$ chosen so that $\int_\mathbb{R} \eta(x) dx = 1$.
Let $\eta_\epsilon(x) \coloneqq \frac{1}{\epsilon}\eta(\frac{x}{\epsilon})$.
Convolving both sides of (3.14) with $\eta_\epsilon$ gives
$$\frac{f_\epsilon(x - \delta -\alpha \delta^2) + f_\epsilon(x + \delta -\alpha \delta^2)}{2} \leq e^{-\alpha\delta^2}f_\epsilon(x)$$
where $f_\epsilon = f*\eta_\epsilon$.

$f_\epsilon \in C^\infty(\mathbb{R})$, so by the proof of the previous theorem, we have $g_\epsilon'' + bg_\epsilon \leq 0$ where $g_\epsilon(x) \coloneqq e^{-\alpha x}f_\epsilon(x)$ and $b \coloneqq 2\alpha - \alpha^2$. As before, we suppose $0 < \alpha < 2$ so that $b>0$.

Fix $\epsilon' > 0$. Select $x_0$ such that $\epsilon' < e^{x_0} - 1$. Since $f$ is continuous by assumption, $f_\epsilon \to f$ uniformly on compact sets \cite{Evans}, and we can select $\epsilon$ such that $\left \vert f - f_\epsilon \right \vert < \epsilon'$ on $[x_0,x_0 + \frac{\pi}{\sqrt{b}}]$. By Lemma 3.4, there exists $x \in [x_0, x_0 + \frac{\pi}{\sqrt{b}}]$ such that $g_\epsilon(x) \leq 0$ which implies $f_\epsilon(x) \leq 0$. Therefore, for this $x$, we have
$$\mathcal{B}(x,0) = f(x) < f_\epsilon(x) + \epsilon' \leq \epsilon' < e^{x_0}-1<e^x-1 = V(x,0)$$

Once again, we've shown $\mathcal{B}$ cannot posses both the majorization and concavity properties of Theorem 3.1, and so (3.3) is false for $\alpha < 2$ assuming $\mathcal{B}$ is continuous.
\end{proof}

Our final task is to justify the assumption that $\mathcal{B}$ is continuous. Indeed, this follows from the concavity property.

\begin{lemma}
Given $\alpha \in \mathbb{R}$, if $f:\mathbb{R} \to \mathbb{R}$ is an increasing function such that
$$\frac{f(x_0 - t -\alpha t^2) + f(x_0 + t -\alpha t^2)}{2} \leq e^{-\alpha t^2}f(x_0)$$
for all $t$, then $f$ is continuous at $x_0$.
\end{lemma}

\begin{proof}
Let $U = \lim\limits_{x \to x_0^+} f(x)$ and $L = \lim\limits_{x \to x_0^-} f(x)$. $U$ and $L$ exists with $L \leq U$ because $f$ is increasing. It suffices to show that $L \geq U$.

Fix $\epsilon > 0$. Choose $\delta > 0$ such that $f(x) > L - \epsilon$ when $x \in (x_0 - \delta, x_0)$. Select $t$ such that
$$0 < t < \min\left(\frac{1}{\alpha},\sqrt{\frac{\delta}{\alpha} + \frac{1}{4}}-\frac{1}{2}\right)$$
Note that this selection of $t$ ensures that $0 < t - \alpha t^2$ and $\alpha t^2 + t < \delta$. Finally, choose $x$ such that
$$\min (x_0 + \alpha t^2 - t, x_0 + \alpha t^2 + t - \delta) < x < x_0$$
Such an $x$ is known to exists since $\alpha t^2 - t < 0$ and $\alpha t^2 + t < \delta$.
Then we have
$$x_0 - \delta < x - \alpha t^2 - t < x < x_0 < x - \alpha t^2 + t$$
and hence
$$\frac{L - \epsilon + U}{2} \leq \frac{f(x - \alpha t^2 - t) + f(x - \alpha t^2 + t)}{2} \leq f(x)e^{-\alpha t} \leq f(x) \leq L$$
Since $\epsilon$ was arbitary, we have $\frac{L + U}{2} \leq L$ and thus $U \leq L$ as desired.
\end{proof}

We are now prepared to state our main result.

\begin{theorem}
$\alpha = 2$ is the smallest constant such that
$$\mathbb{E}e^{f_n} \leq \mathbb{E}e^{\alpha(Sf_n)^2}$$
for all dyadic martingales with $f_0 = 0$.
\end{theorem}

\begin{proof}
It is clear from the definition that $\mathcal{B}(x,0)$ is increasing. This fact coupled with the concavity property implies $\mathcal{B}(x,0)$ is continuous by the previous lemma. Now apply Theorem 3.6.
\end{proof}

\section{Examples and Future Work}
It is worth noting that a simpler proof of Theorem 3.2 exists. It uses a result known as Rubin's lemma \cite{Pipher}, which says that if $\{f_n\}$ is a real valued dyadic martingale on $[0,1)$ whose limit is $f$, then for all $\lambda \geq 0$,
$$\int_0^1 e^{\lambda\left( f(x) -\langle f \rangle_{[0,1)}\right) - \frac{\lambda^2}{2}\left(Sf\right)^2(x)}dx \leq 1$$
See \cite{BM} for a proof.

We can use the Cauchy-Schwarz inequality with Rublin's lemma to prove Theorem 3.2. As usual, we may work on $([0,1),\mathcal{B},|\cdot|)$.
\begin{align*}
\int_0^1 e^{f(x)-\langle f \rangle_{[0,1)}}dx &= \int_0^1 e^{f(x)-\langle f \rangle_{[0,1)}-(Sf)^2(x)+(Sf)^2(x)}dx &\\
&\leq \sqrt{\int_0^1 e^{2[f(x)-\langle f \rangle_{[0,1)}]-2(Sf)^2(x)}dx} \sqrt{\int_0^1 e^{2(Sf)^2(x)}dx} &\\
&\leq \sqrt{1} \sqrt{\int_0^1 e^{2(Sf)^2(x)}dx} &\\
&\leq \int_0^1 e^{2 (Sf)^2(x)}dx &
\end{align*}

The Bellman function method is still desirable in at least two ways. First, it's a general technique for proving inequalities, while this simpler proof is very particular to the details of our problem. Second, the Bellman function allowed us to prove the sharpness of Theorem 3.2. The shorter proof would requiring an accompanying construction of a martingale which maximizes the left side of the inequality relative to the right side.

Such a construction could provide the basis for some future work. Sometimes, extremal examples can be deduced from the Bellman function itself (see, e.g. \cite{Wang}). Although we were unable to do this, we have an example of a martingale which shows that
\begin{equation}
\int_0^1 e^{f(x)-\langle f \rangle_{[0,1)}}dx \leq \int_0^1 e^{\alpha (Sf)^2(x)}dx
\end{equation}
is false for $\alpha < \log_2(e) \approx 1.44$. Recall that $\alpha = 2$ was optimal.

Let $f=\sum\limits_{n=0}^\infty \chi_{I_n^0} = \sum\limits_{n=1}^\infty n\chi_{I_n^1}$ where $I_n^k = \left[\frac{k}{2^n},\frac{k+1}{2^n}\right)$. As always, we can think of $f$ as a dyadic martingale by letting $f_n = \mathbb{E}[f\vert \mathcal{D}_n]$ where $\mathcal{D}_n$ is the $\sigma$-algebra generated by the nth generation of dyadic subintervals of $[0,1)$. This function is a discrete approximation of $-\log_2x$ has the property that $(Sf)^2 = f$. It is a "fixed point" of the operator that sends $f \mapsto (Sf)^2$. The left side of (4.1) becomes
$$\int_0^1 e^{f(x)-\langle f \rangle_{[0,1)}}dx \approx \int_0^1 e^{-\log_2x}dx = \int_0^1 x^{-\log_2(e)}dx = \infty$$
On the other hand, the right side becomes
$$\int_0^1 e^{\alpha (Sf)^2(x)}dx = \int_0^1 e^{\alpha f(x)}dx \approx \int_0^1 e^{-\alpha\log_2x}dx = \int_0^1 x^{-\alpha\log_2(e)}dx$$
which is finite for $\alpha < \frac{1}{\log_2(e)} = \ln(2) \approx .69$. So our example falsifies (4.1) for this range of $\alpha$.

We can push this example further. $[S(\lambda g)]^2 = \lambda^2 [S(g)]^2$ for all $g$ follows immediately from the definition of $S$. Applying this to our present example gives $[S(\lambda f)]^2 = \lambda^2 [S(f)]^2 = \lambda^2 f$. Plugging $\lambda f$ into (4.1), the left side is approximately $\int_0^1 x^{-\lambda\log_2(e)}dx$ and the right side is approximately $\int_0^1 x^{-\alpha\lambda^2\log_2(e)}dx$. Thus for $\lambda = \frac{1}{\log_2(e)}$, the left side is infinite and the right side is finite when $\alpha < \log_2(e) \approx 1.44$, so $\frac{f}{\log_2(e)}$ falsifies (4.1) for this range of $\alpha$.

We are still left without a martingale showing $\alpha = 2$ is sharp. While the Bellman function proof makes it unnecessary, finding an explicit example is an interesting future problem.
\bibliographystyle{ieeetr}
\bibliography{bibliography}

\end{document}